\documentclass[11pt,reqno]{amsart}
 \usepackage{amssymb,amsfonts,amscd,amsbsy}
\usepackage[mathscr]{eucal}
\usepackage{url}

\theoremstyle{plain}
\newtheorem{thm}{Theorem}[section]

\newtheorem{cor}[thm]{Corollary}
\theoremstyle{definition}

\numberwithin{equation}{section}

\begin{document}
\title{Quadratic forms and four partition functions modulo $3$}
\author{Jeremy Lovejoy and Robert Osburn}

\address{CNRS, LIAFA, Universit{\'e} Denis Diderot - Paris 7, Case 7014, 75205 Paris Cedex 13, FRANCE}

\address{School of Mathematical Sciences, University College Dublin, Belfield, Dublin 4, Ireland}

\email{lovejoy@liafa.jussieu.fr}

\email{robert.osburn@ucd.ie}

\subjclass[2000]{Primary: 11P83; Secondary: 11E25}
\keywords{partitions, overpartitions, congruences, binary quadratic forms, sums of squares}

\date{\today}

\begin{abstract}
Recently, Andrews, Hirschhorn and Sellers have proven congruences modulo $3$ for four types of partitions using elementary series manipulations.  In this paper, we generalize their congruences using arithmetic properties of certain quadratic forms.
\end{abstract}

\maketitle
\section{Introduction}
A partition of a non-negative integer $n$ is a non-increasing sequence whose sum is $n$. An overpartition of $n$ is a partition of $n$ where we may overline the first occurrence of a part. Let $\overline{p}(n)$ denote the number of overpartitions of $n$, $\overline{p_{o}}(n)$ the number of overpartitions of $n$ into odd parts, $ped(n)$ the number of partitions of $n$ without repeated even parts and $pod(n)$ the number of partitions of $n$ without repeated odd parts.  The generating functions for these partitions are
\begin{eqnarray}
\sum_{n \geq 0} \overline{p}(n)q^n &=& \frac{(-q;q)_{\infty}}{(q;q)_{\infty}}, \label{gf1}\\
\sum_{n \geq 0} \overline{p_o}(n)q^n &=& \frac{(-q;q^2)_{\infty}}{(q;q^2)_{\infty}}, \label{gf2}\\
\sum_{n \geq 0} ped(n)q^n &=& \frac{(-q^2;q^2)_{\infty}}{(q;q^2)_{\infty}}, \label{gf3}\\
\sum_{n \geq 0} pod(n)q^n &=& \frac{(-q;q^2)_{\infty}}{(q^2;q^2)_{\infty}} \label{gf4},
\end{eqnarray} 
where as usual 
$$
(a;q)_n := (1-a)(1-aq) \cdots (1-aq^{n-1}).
$$
The infinite products in \eqref{gf1}--\eqref{gf4} are essentially the four different ways one can specialize the product $(-aq;q)_{\infty}/(bq;q)_{\infty}$ to obtain a modular form whose level is relatively prime to $3$.  

A series of four recent papers examined congruence properties for these partition functions modulo $3$ \cite{An-Hi-Se1,Hi-Se1,Hi-Se2,Hi-Se3}.  Among the main theorems in these papers are the following congruences (see Theorem 1.3 in \cite {Hi-Se2}, Corollary 3.3 and Theorem 3.5 in \cite{An-Hi-Se1}, Theorem 1.1 in \cite{Hi-Se1} and Theorem 3.2 in \cite{Hi-Se3}, respectively).  For all $n \geq 0$ and $\alpha \geq 0$ we have

\begin{equation} \label{first}
\overline{p_o}(3^{2 \alpha}(An+B)) \equiv 0 \pmod{3},
\end{equation}

\noindent where $An+B = 9n+6$ or $27n+9$,

\begin{equation} \label{second1} 
ped\Bigl(3^{2\alpha + 3} n + \frac{17 \cdot 3^{2\alpha+2} - 1}{8}\Bigr) \equiv ped\Bigl(3^{2\alpha+2} n + \frac{19 \cdot 3^{2\alpha+1} - 1}{8}\Bigr) \equiv 0 \pmod 3,
\end{equation}


\begin{equation} \label{third}
\overline{p}(3^{2\alpha}(27n+18)) \equiv 0 \pmod{3}
\end{equation}

\noindent and

\begin{equation} \label{fourth}
pod\left(3^{2\alpha+3} + \frac{23\cdot3^{2\alpha+2} + 1}{8}\right) \equiv 0 \pmod{3}.
\end{equation} 

We note that congruences modulo $3$ for $\overline{p}(n)$, $\overline{p}_o(n)$ and $ped(n)$ are typically valid modulo $6$ or $12$.  The powers of $2$ enter trivially (or nearly so), however, so we do not mention them here. 

The congruences in \eqref{first}--\eqref{fourth} are proven in \cite{An-Hi-Se1,Hi-Se1,Hi-Se2,Hi-Se3} using elementary series manipulations.   If we allow ourselves some elementary number theory, we find that much more is true.  

With our first result we exhibit formulas for $\overline{p}_o(3n)$ and $ped(3n+1)$ modulo $3$ for all $n \geq 0$.  These formulas depend on the factorization of $n$, which we write as 
\begin{equation} \label{factor}
n=2^{a} 3^{b} \prod_{i=1}^{r} p_{i}^{v_{i}} \prod_{j=1}^{s} q_{j}^{w_{j}},
\end{equation}

\noindent where $p_{i} \equiv 1$, $5$, $7$ or $11 \pmod{24}$ and $q_{j} \equiv 13$, $17$, $19$ or $23 \pmod{24}$.  Further, let $t$ denote the number of prime factors of $n$ (counting multiplicity) that are congruent to $5$ or $11 \pmod{24}$. Let $R(n, Q)$ denote the number of representations of $n$ by the quadratic form $Q$.  

\begin{thm} \label{ourfirst}
For all $n \geq 0$ we have 
\begin{equation*}
\overline{p}_o(3n) \equiv f(n)R(n,x^2+6y^2) \pmod{3} 
\end{equation*}
\noindent and 
\begin{equation*}
ped(3n+1) \equiv (-1)^{n+1}R(8n+3,2x^2+3y^2) \pmod{3},
\end{equation*}
where $f(n)$ is defined by 
\begin{equation*}
f(n) = 
\begin{cases}
-1, & \text{$n \equiv 1,6,9,10 \pmod{12}$},\\
1, & \text{otherwise}.  
\end{cases}
\end{equation*}
\noindent Moreover, we have  
\begin{equation}  \label{formula1}
\overline{p}_o(3n)  \equiv f(n)(1 + (-1)^{a+b+t}) \prod_{i=1}^{r} (1+v_{i}) \prod_{j=1}^{s} \Biggl( \frac{1 + (-1)^{w_{j}}}{2} \Biggr) \pmod{3}
\end{equation}
\noindent and
\begin{equation} \label{formula2}
(-1)^{n}ped(3n+1) \equiv \overline{p}_o(48n+18) \pmod{3}.
\end{equation}
\end{thm}

There are many ways to deduce congruences from Theorem \ref{ourfirst}.  For example, calculating the possible residues of $x^2+6y^2$ modulo $9$ we see that 

$$R(3n+2,x^2+6y^2) = R(9n+3,x^2+6y^2) = 0,$$ 

\noindent and then \eqref{formula1} implies that $\overline{p}_o(27n) \equiv  \overline{p}_o(3n) \pmod{3}$.  This gives \eqref{first}.  The congruences in \eqref{second1} follow from those in \eqref{first} after replacing $48n+18$ by $3^{2\alpha}(48(3n+2) + 18)$ and $3^{2\alpha}(48(9n+6) + 18)$ in \eqref{formula2}.  We record two more corollaries, which also follow readily from Theorem \ref{ourfirst}.

\begin{cor} \label{cor1}
For all $n \geq 0$ and $\alpha \geq 0$ we have
\begin{equation*}
\overline{p_o}(2^{2\alpha}(An+B)) \equiv 0 \pmod{3},
\end{equation*}
where $An+B = 24n+9$ or $24n+15$.
\end{cor}  

\begin{cor} \label{cor1bis}
If $\ell \equiv 1,5,7$ or $11 \pmod{24}$ is prime, then for all $n$ with $\ell \nmid n$ we have
\begin{equation}
\overline{p}_o(3 \ell^2 n) \equiv 0 \pmod{3}.
\end{equation}
\end{cor}

For the functions $\overline{p}(3n)$ and $pod(3n+2)$ we have relations not to binary quadratic forms but to $r_5(n)$, the number of representations of $n$ as the sum of five squares. Our second result is the following.

\begin{thm} \label{oursecond}
For all $n \geq 0$ we have 
$$\overline{p}(3n) \equiv (-1)^nr_5(n) \pmod{3}$$ 
\noindent and 
$$pod(3n+2) \equiv (-1)^{n}r_5(8n+5) \pmod{3}.$$
\noindent Moreover, for all odd primes $\ell$ and $n \geq 0$, we have
\begin{equation} \label{Hecke1}
\overline{p}(3\ell^2n) \equiv \left(\ell - \ell\left(\frac{n}{\ell}\right) + 1\right)\overline{p}(3n) - \ell  \overline{p}(3n/\ell^2) \pmod{3}
\end{equation}
\noindent and
\begin{equation} \label{relation}
(-1)^{n+1}pod(3n+2)  \equiv \overline{p}(24n+15) \pmod{3},
\end{equation}
\noindent where $\left(\frac{\bullet}{\ell}\right)$ denotes the Legendre symbol.
\end{thm}

Here we have taken $\overline{p}(3n/\ell^2)$ to be $0$ unless $\ell^2 \mid 3n$. Again there are many ways to deduce congruences.  For example, (\ref{third}) follows readily upon combining \eqref{Hecke1} in the case $\ell = 3$ with the fact that 
$$r_5(9n+6) \equiv 0 \pmod{3},$$
\noindent which is a consequence of the fact that $R(9n+6,x^2+y^2+3z^2) = 0$.  One can check that (\ref{fourth}) follows similarly.  For another example, we may apply \eqref{Hecke1} with $n$ replaced by $n\ell$ for $\ell \equiv 2 \pmod{3}$ to obtain

\begin{cor} \label{cor2}
If $\ell \equiv 2 \pmod{3}$ is prime and $\ell \nmid n$, then
\begin{equation*}
\overline{p}(3\ell^3n) \equiv 0 \pmod{3}.
\end{equation*}
\end{cor}


\section{Proofs of Theorems \ref{ourfirst} and \ref{oursecond}}
\begin{proof}[Proof of Theorem \ref{ourfirst}]
On page 364 of \cite{Hi-Se2} we find the identity 
$$
\sum_{n \geq 0} \overline{p}_o(3n) q^n = \frac{D(q^3)D(q^6)}{D(q)^2},
$$
\noindent where 
$$
D(q) := \sum_{n \in \mathbb{Z}} (-1)^nq^{n^2}.
$$
\noindent Reducing modulo $3$, this implies that
\begin{eqnarray*}
\sum_{n \geq 0} \overline{p_o}(3n)q^n &\equiv& \sum_{x,y \in \mathbb{Z}} (-1)^{x+y}q^{x^2+6y^2} \pmod{3} \\
&\equiv & \sum_{n \geq 0} f(n)R(n,x^2+6y^2)q^n \pmod{3}.
\end{eqnarray*}
\noindent Now it is known (see Corollary 4.2 of \cite{Be-Ye1}, for example) that if $n$ has the factorization in \eqref{factor}, then 

\begin{equation} \label{x26y2}
R(n,x^2+6y^2)=(1 + (-1)^{a+b+t}) \prod_{i=1}^{r} (1+v_{i}) \prod_{j=1}^{s} \Biggl( \frac{1 + (-1)^{w_{j}}}{2} \Biggr).
\end{equation}
\noindent This gives \eqref{formula1}. Next, from \cite{An-Hi-Se1} we find the identity
$$
\sum_{n \geq 0} ped(3n+1)q^n = \frac{D(q^3)\psi(-q^3)}{D(q)^2},
$$
\noindent where 
$$
\psi(q) := \sum_{n \geq 0} q^{n(n+1)/2}.
$$
\noindent Reducing modulo $3$, replacing $q$ by $-q^8$ and multiplying by $q^3$ gives 
$$
\sum_{n \geq 0} (-1)^{n+1} ped(3n+1) q^{8n+3} \equiv \sum_{n \geq 0} R(8n+3, 2x^2+3y^2) q^{8n+3} \pmod 3.
$$
\noindent It is known (see Corollary 4.3 of \cite{Be-Ye1}, for example) that if $n$ has the factorization given in \eqref{factor}, then 

\begin{equation*}
R(n,2x^2+3y^2)=(1 - (-1)^{a+b+t}) \prod_{i=1}^{r} (1+v_{i}) \prod_{j=1}^{s} \Biggl( \frac{1 + (-1)^{w_{j}}}{2} \Biggr).
\end{equation*}
\noindent Comparing with \eqref{x26y2} finishes the proof of (\ref{formula2}).
\end{proof}

\begin{proof}[Proof of Theorem \ref{oursecond}] 
On page 3 of \cite{Hi-Se1} we find the identity
$$
\sum_{n \geq 0} \overline{p}(3n) q^n \equiv \frac{D(q^3)^2}{D(q)} \pmod{3}.
$$
Reducing modulo $3$ and replacing $q$ by $-q$ yields
\begin{equation*}
\sum_{n \geq 0} (-1)^n\overline{p}(3n)q^n \equiv \sum_{n \geq 0} r_5(n)q^n \pmod{3}.  
\end{equation*}
It is known (see Lemma 1 in \cite{Co1}, for example) that for any odd prime $\ell$ we have
\begin{equation*} \label{Heckeop}
r_5(\ell^2n) = \left(\ell^3 - \ell\left(\frac{n}{\ell}\right) + 1\right)r_5(n) - \ell^3r_5(n/\ell^2).
\end{equation*}
Here $r_5(n/\ell^2) = 0$ unless $\ell^2 \mid n$.  Replacing $r_5(n)$ by $(-1)^n\overline{p}(3n)$ 
throughout gives \eqref{Hecke1}. Now equation $(1)$ of \cite{Hi-Se3} reads
$$
\sum_{n \geq 0}(-1)^n pod(3n+2) q^n  = \frac{\psi(q^3)^3}{\psi(q)^4}.
$$
Reducing modulo $3$ we have
\begin{eqnarray*}
\sum_{n \geq 0} (-1)^n pod(3n+2) q^n &\equiv& \psi(q)^5 \pmod 3 \\
&\equiv& \sum_{n \geq 0} r_5(8n+5)q^n \pmod{3} \\
&\equiv& - \sum_{n \geq 0} \overline{p}(24n+15) q^n \pmod{3},
\end{eqnarray*}
where the second congruence follows from Theorem 1.1 in \cite{Ba-Co-Hi1}. This implies (\ref{relation}) and thus the proof of Theorem \ref{oursecond} is complete.

\end{proof}

\section*{Acknowledgement}
We would like to thank Scott Ahlgren for pointing out reference \cite{Co1}. The second author was partially funded by Science Foundation Ireland 08/RFP/MTH1081.

 \end{document}